\documentclass[a4paper,twoside,11pt,notitlepage]{amsart}
\usepackage{a4wide}
\usepackage[T1]{fontenc}
\usepackage{amsthm}
\usepackage[latin1]{inputenc}
\usepackage{amsfonts,amsmath,epsfig,amscd,amssymb,latexsym}

\theoremstyle{plain}
\newtheorem{theorem}{Theorem}[section]
\newtheorem{lemma}[theorem]{Lemma}

\theoremstyle{definition}


\title{On the mean square of short exponential sums related to cusp forms}
\author{Anne-Maria Ernvall-Hyt\"onen}
\address{Department of Mathematics\\
         Kungliga Tekniska H\"ogskolan\\
         Sweden}
\thanks{The author is funded by the Swedish research council Vetenskapsr{\aa}det (grant 2009-721). Also, the author would like to thank professor Shparlinski for suggesting this problem.}
\newcommand{\ud}{\textup{d}}

\begin{document}\maketitle
\begin{abstract} The purpose of the article is to estimate the mean square of a squareroot length exponential sum of Fourier coefficients of a holomorphic cusp form.
\end{abstract}
\section{Introduction}
Let $f(z)=\sum_{N\geq 1}a(n)n^{(\kappa-1)/2}e(nz)$ be a holomorphic cusp form of weight $\kappa$ with respect to the full modular group. Long exponential sums
\[
\sum_{1\leq n\leq M}a(n)e(n\alpha),
\]
where $\alpha$ is a real number, have been widely studied. See e.g. Wilton \cite{wilton} and Jutila \cite{jutila:ramanujan}. Short sums
\[
\sum_{M\leq n\leq M+\Delta}a(n)e(n\alpha),
\]
where $\Delta\ll M^{3/4}$ have been studied for instance in \cite{e&k} and \cite{oma:parannus}. However, it seems that very short sums, in particular, sums with $\Delta\asymp M^{1/2}$ seem to be extremely difficult to treat, even though this is an important special case. According to the results in \cite{oma:weird} and the computer data in \cite{oma:laskenta}, it is plausible to believe the correct upper bound to be
\[
\sum_{M\leq n\leq M+\sqrt{M}}a(n)e(n\alpha)\ll M^{1/4+\varepsilon}.
\]
However, anything like this seems to be hopeless to get by at the moment, and therefore, in the current paper, the aim is to consider the mean square of the sum in rational points. The mean square is a common way to consider sums that seem difficult to come by. See e.g. Jutila \cite{jutila:divisor} or Ivic \cite{ivic:meansquaredivisor}. We prove the following theorem which shows this conjecture to be true in average:
\begin{theorem}\label{main} Let $h$ and $k$, $0\leq h<k\ll M^{1/4}$, be integers. Let $w(x)$ denote a smooth weight function that is supported on the interval $[M,M+\Delta]$ where $kM^{1/2+\delta}\ll \Delta \ll M$ with $\delta$ an arbitrarily small fixed positive real number. Further assume that $w(M)=w(M+\Delta)=0$, $0\leq w(x)\leq 1$,  and $w^{(n)}(x)\ll \Delta^{-n}$ for $1\leq n\leq J$ for a sufficiently large $J$ depending on $\delta$. Then
\[
\int_M^{M+\Delta}\left|\sum_{x\leq n\leq x+\sqrt{x}}a(n)e\left(\frac{hn}{k}\right)\right|^2w(x)\ll k^{\varepsilon}\Delta M^{1/2},
\]
where the constant implied by the $\ll$ symbol depends only on $\varepsilon$.
\end{theorem}
On the other hand, the Omega results in \cite{oma:weird} and \cite{ivic:short} show that
\[
\sum_{M\leq n\leq M+\Delta}a(n)=\Omega(\sqrt{\Delta}),
\]
where $f=\Omega(g)$ is to be understood to mean that $f=o(g)$ does not hold.

Throughout the paper, $\varepsilon$ denotes a real number which can be chosen to be arbitrarily small, however, $\varepsilon$ is not necessarily the same at every incidence.  The constants implied by $\ll$ depend only on $\varepsilon$. Also, let $w(x)$ denote a smooth weight function that is defined as in Theorem \ref{main}.

\section{Lemmas}
The following slightly modified version of Jutila and
Motohashi's Lemma 6 in \cite{jutimoto:acta} is extremely useful while estimating oscillating integrals. The proof is similar to
the proof of the original lemma.

\begin{lemma}\label{jutilamotohashi}
Let $A$ be a $P\geq 0$ times differentiable function which is
compactly supported in a finite interval $[a,b]$. Assume also that
there exist two quantities $A_0$ and $A_1$ such that for any
non-negative integer $\nu\leq P$ and for any $x\in [a,b]$,
\[
A^{(\nu)}(x)\ll A_0A_1^{-\nu}.
\]
Moreover, let $B$ be a function which is real-valued on $[a,b]$,
and regular throughout the complex domain composed of all points
within the distance $\varrho$ from the interval; and assume that
there exists a quantity $B_1$ such that
\[
0<B_1\ll \left|B'(x)\right|
\]
for any point $x$ in the domain. Then we have
\[
\int_{a}^{b}A(x)e\left(B(x)\right)\ud x \ll
A_0\left(A_1B_1\right)^{-P}\left(1+\frac{A_1}{\varrho}\right)^P\left(b-a\right).
\]
\end{lemma}

\begin{lemma}\label{diagonaali} Let $0\leq h< k\leq M^{1/4}$. Now
\begin{multline*}
\frac{k}{2\pi^2}\sum_{n\leq M}\frac{\left|a(n)\right|^2}{n^{3/2}}\int_M^{M+\Delta}w(x)x^{1/2}\left(\cos \left(\frac{4\pi\sqrt{n(x+\sqrt{x})}}{k}-\frac{\pi}{4}\right)-\cos\left(4\pi\frac{\sqrt{nx}}{k}-\frac{\pi}{4}\right)\right)^2\ud x\\ \ll k^{\varepsilon}\Delta M^{1/2}
\end{multline*}
\end{lemma}

\begin{proof}
Notice first that
\begin{multline*}
\left(\cos \left(\frac{4\pi\sqrt{n(x+\sqrt{x})}}{k}-\frac{\pi}{4}\right)-\cos\left(4\pi\frac{\sqrt{nx}}{k}-\frac{\pi}{4}\right)\right)^2\\=\sin^2\left(2\pi\frac{\sqrt{n(x+\sqrt{x})}}{k}+2\pi\frac{\sqrt{nx}}{k}-\frac{\pi}{4}\right)\sin^2\left(2\pi\frac{\sqrt{n(x+\sqrt{x})}}{k}-2\pi\frac{\sqrt{nx}}{k}\right).
\end{multline*}
Since $\frac{\sqrt{n(x+\sqrt{x})}}{k}-\frac{\sqrt{nx}}{k}\ll \frac{\sqrt{n}}{k}$, we have
\begin{multline*}
\sin^2\left(2\pi\frac{\sqrt{n(x+\sqrt{x})}}{k}+2\pi\frac{\sqrt{nx}}{k}-\frac{\pi}{4}\right)\sin^2\left(2\pi\frac{\sqrt{n(x+\sqrt{x})}}{k}-2\pi\frac{\sqrt{nx}}{k}\right)\\ \ll \sin^2\left(2\pi\frac{\sqrt{n(x+\sqrt{x})}}{k}+2\pi\frac{\sqrt{nx}}{k}-\frac{\pi}{4}\right)\cdot \frac{n}{k^2},
\end{multline*}
when $n\leq k^2$. For $n>k$, estimate the sine-part of the integral to be $\leq 1$. We obtain
\begin{multline*}
\frac{k}{2\pi^2}\sum_{n\leq M}\left|a(n)\right|^2n^{-3/2}\int_M^{M+\Delta}w(x)x^{1/2}\left(\cos \left(\frac{4\pi\sqrt{n(x+\sqrt{x})}}{k}-\frac{\pi}{4}\right)-\cos\left(4\pi\frac{\sqrt{nx}}{k}-\frac{\pi}{4}\right)\right)^2\ud x\\ \ll k\sum_{n\leq k^2}n^{\varepsilon-3/2+1}k^{-2}\int_M^{M+\Delta}w(x)x^{1/2}\ud x+k\sum_{k^2\leq n\leq M}n^{\varepsilon-3/2}\int_{M}^{M+\Delta}w(x)x^{1/2}\ud x \ll k^{\varepsilon}\Delta M^{1/2}
\end{multline*}
\end{proof}
Using Lemma \ref{jutilamotohashi}, we get the following estimates
\begin{lemma}\label{erifunktiot} Let $1\leq m,n\leq M$.  Further assume $0\leq h< k\leq M^{1/4}$. Then
\[
\int_M^{M+\Delta} w(x)x^{1/2}e\left(\pm \left(2\frac{\sqrt{nT_1(x)}}{k}+ 2\frac{\sqrt{mT_2(x)}}{k}\right)\right)\ud x\ll \left(\sqrt{n}+\sqrt{m}\right)^{-P}\Delta^{1-P}k^PM^{P/2},
\]
where $T_1(x)$ and $T_2(x)$ are $x$ or $x+\sqrt{x}$ (not necessarily but possibly the same).
\end{lemma}
\begin{lemma}\label{erimerkki} Let $1\leq m,n\leq M$.  Further assume $0\leq h< k\leq M^{1/4}$. Then
\[
\int_M^{M+\Delta} w(x)x^{1/2}e\left(\pm 2\frac{\sqrt{nT(x)}}{k}\mp 2\frac{\sqrt{mT(x)}}{k}\right)\ud x\ll \left(\sqrt{n}-\sqrt{m}\right)^{-P}\Delta^{1-P}k^PM^{P/2},
\]
where $T(x)=x$ or $T(x)=x+\sqrt{x}$.
\end{lemma}
\begin{lemma}\label{hankalin} Let $1\leq m,n\leq M$.  Further assume $0\leq h< k\leq M^{1/4}$. Then
\[
\int_M^{M+\Delta}x^{1/2}w(x)e\left(\pm2\frac{\sqrt{m(x+\sqrt{x})}}{k}\mp2\frac{\sqrt{nx}}{k}\right)\ud x\ll \Delta^{1-P}\left|\sqrt{m}-\sqrt{n}\right|^{-P}k^PM^{P/2}.
\]
\end{lemma}
\begin{proof} When $m>n$, the proof is similar to Lemma \ref{erimerkki}. Therefore, it is sufficient to concentrate on the case $n>m$. We may also assume the first sign to be plus, and the second one to be minus, as the other case can be treated similarly. Write
\[
B(x)=2\frac{\sqrt{m(x+\sqrt{x})}}{k}-2\frac{\sqrt{nx}}{k}.
\]
Now
\begin{multline*}
B'(x)=\frac{\sqrt{m}}{k\sqrt{(x+\sqrt{x})}}\left(1+\frac{1}{2}x^{-1/2}\right)-\frac{\sqrt{n}}{k\sqrt{x}}=\frac{\sqrt{m}}{k}\left(\frac{1+x^{-1/2}+\frac{1}{4}x^{-1}}{x+\sqrt{x}}\right)^{1/2}-\frac{\sqrt{n}}{k\sqrt{x}}\\=\frac{\sqrt{m}}{k\sqrt{x}}\left(1+\frac{1}{4(x+\sqrt{x})}\right)^{1/2}-\frac{\sqrt{n}}{k\sqrt{x}}=\frac{\sqrt{m}}{k\sqrt{x}}\left(1+\frac{1}{8(x+\sqrt{x})}+O\left(\frac{1}{x^2}\right)\right)-\frac{\sqrt{n}}{k\sqrt{x}} \\=\frac{\sqrt{m}-\sqrt{n}}{k\sqrt{x}}+\frac{\sqrt{m}}{8k\sqrt{x}(x+\sqrt{x})}+O\left(\frac{\sqrt{m}}{kx^{5/2}}\right).
\end{multline*}
Let us now estimate the second term and the error term. When $x$ is sufficiently large, we have
\[
\frac{\sqrt{m}}{8k\sqrt{x}(x+\sqrt{x})}+O\left(\frac{\sqrt{m}}{kx^{5/2}}\right)\leq \frac{1}{4kx}\leq\left|\frac{\sqrt{m}-\sqrt{n}}{4k\sqrt{x}}\right|.
\]
Therefore,
\[
\left|B'(x)\right|=\left|\frac{\sqrt{m}}{k\sqrt{m(x+\sqrt{x})}}\left(1+\frac{1}{2}x^{-1/2}\right)-\frac{\sqrt{n}}{k\sqrt{x}}\right|\geq  \left|3\frac{(\sqrt{m}-\sqrt{n})}{4k\sqrt{x}}\right|.
\]
Using Lemma \ref{jutilamotohashi} we obtain the estimate
\[
\int_M^{M+\Delta}x^{1/2}w(x)e\left(\pm\left(2\frac{\sqrt{m(x+\sqrt{x})}}{k}-2\frac{\sqrt{nx}}{k}\right)\right)\ud x\ll \Delta^{1-P}\left|\sqrt{m}-\sqrt{n}\right|^{-P}k^PM^{P/2}.
\]
as desired.
\end{proof}
\section{Proof of the main theorem}
Let us first use a modification of Theorem 1.1 \cite{jutila:lectures} (proof is similar than that of the original theorem, just the Fourier coefficients have been normalized):
\begin{multline*}
\sum_{1\leq n\leq x}a(n)e\left(\frac{hn}{k}\right)=\left(\pi\sqrt{2}\right)^{-1}k^{1/2}x^{1/4}\sum_{n\leq N}a(n)e_k(-n\bar{h})n^{-3/4}\cos\left(\frac{4\pi\sqrt{nx}}{k}\right)\\ +O\left(kx^{1/2+\varepsilon} N^{-1/2}\right).
\end{multline*}
Choose $N\asymp x$. Now
\begin{multline*}
\sum_{x\leq n\leq x+\sqrt{x}}a(n)e\left(\frac{hn}{k}\right)=\left(\pi\sqrt{2}\right)^{-1}k^{1/2}\\ \times\sum_{n\leq x}a(n)e_k(-n\bar{h})n^{-3/4}\left(\cos\left(\frac{4\pi\sqrt{n(x+\sqrt{x})}}{k}\right)\left(x+\sqrt{x}\right)^{1/4}-\cos\left(\frac{4\pi\sqrt{nx}}{k}\right)x^{1/4}\right)+ O\left(kx^{\varepsilon}\right)\\ =\frac{k^{1/2}}{\pi\sqrt{2}}\sum_{n\leq x}a(n)e_k(-n\bar{h})n^{-3/4}x^{1/4}\left(\cos\left(\frac{4\pi\sqrt{n(x+\sqrt{x})}}{k}\right)-\cos\left(\frac{4\pi\sqrt{nx}}{k}\right)\right) +O\left(kx^{\varepsilon}\right).
\end{multline*}
Now
\begin{multline*}
\int_M^{M+\Delta}\left|\sum_{x\leq n\leq x+\sqrt{x}}a(n)e\left(\frac{hn}{k}\right)\right|^2w(x)\ud x\\ =  \frac{k}{2\pi^2}\int_M^{M+\Delta}\left|\sum_{n\leq x}a(n)e_k(-n\bar{h})n^{-3/4}x^{1/4}\left(\cos\left(\frac{4\pi\sqrt{n(x+\sqrt{x})}}{k}\right)-\cos\left(\frac{4\pi\sqrt{nx}}{k}\right)\right)\right|^2w(x)\ud x\\+O\left(k^2\Delta M^{\varepsilon}\right) \\=\frac{k}{2\pi^2}\sum_{m\ne n}\frac{a(n)a(m)}{(nm)^{3/4}}e_k\left(-n\bar{h}+m\bar{h}\right)\int_M^{M+\Delta} x^{1/2}w(x)\left(\cos\left(\frac{4\pi\sqrt{n(x+\sqrt{x})}}{k}\right)-\cos\left(\frac{4\pi\sqrt{nx}}{k}\right)\right)\\ \times\left(\cos\left(\frac{4\pi\sqrt{m(x+\sqrt{x})}}{k}\right)-\cos\left(\frac{4\pi\sqrt{mx}}{k}\right)\right)\ud x\\ + \frac{k}{2\pi^2}\sum_{n\leq M}\frac{\left|a(n)\right|^2}{n^{3/2}}\int_M^{M+\Delta}w(x)x^{1/2}\left(\cos \left(\frac{4\pi\sqrt{n(x+\sqrt{x})}}{k}-\frac{\pi}{4}\right)-\cos\left(4\pi\frac{\sqrt{nx}}{k}-\frac{\pi}{4}\right)\right)^2\ud x \\+ O\left(k^2\Delta M^{\varepsilon}\right) 
\end{multline*}
The second sum (containing the diagonal terms) has been treated in Lemma \ref{diagonaali}. The cosines in the integral in the first sum can be written as exponential terms. The integrals arising from this have been treated in Lemmas \ref{erimerkki}, \ref{erifunktiot} and \ref{hankalin}. It is therefore sufficient to estimate the sum over the terms estimated in Lemma \ref{hankalin} as all the other sums go similarly. Choose $P=1$. Now
\begin{multline*}
\frac{k}{2\pi^2}\sum_{1\leq m\ne n\leq M}\frac{\left|a(n)a(m)\right|}{(nm)^{3/4}}\left|\sqrt{m}-\sqrt{n}\right|^{-1}kM^{1/2}\\ \ll    k^{2}M^{1/2}\sum_{1\leq m< n\leq M}n^{\varepsilon-1/4}m^{\varepsilon-3/4}\left|m-n\right|^{-1} \ll k^2M^{1/2+\varepsilon}
\end{multline*}
for a suitable choice of $P$. This proves the theorem.


\begin{thebibliography}{10}

\bibitem{oma:weird}
A.-M. Ernvall-Hyt\"onen.
\newblock A relation between {F}ourier coefficients of holomorphic cusp forms
  and exponential sums.
\newblock {\em Publications de l'Institut Mathematique (Beograd)},
  86(100):97--105, 2009.

\bibitem{oma:laskenta}
A.-M. Ernvall-Hyt\"onen and L.~A.
\newblock Bounds and computational results for exponential sums related to cusp
  forms.
\newblock {\em Acta Mathematica Universitatis Ostraviensis}, 17:81--90, 2009.

\bibitem{e&k}
A.-M. Ernvall-Hyt{\"o}nen and K.~Karppinen.
\newblock On short exponential sums involving {F}ourier coefficients of
  holomorphic cusp forms.
\newblock {\em Int. Math. Res. Not. IMRN}, (10):Art. ID. rnn022, 44, 2008.

\bibitem{oma:parannus}
A.-M. Ernvall-Hytönen.
\newblock An improvement on the upper bound of exponential sums of holomorphic
  cusp forms.
\newblock submitted.

\bibitem{ivic:short}
A.~Ivi{\'c}.
\newblock On the divisor function and the {R}iemann zeta-function in short
  intervals.
\newblock {\em Ramanujan J.}, 19(2):207--224, 2009.

\bibitem{ivic:meansquaredivisor}
A.~Ivi{\'c}.
\newblock On the mean square of the divisor function in short intervals.
\newblock {\em J. Th\'eor. Nombres Bordeaux}, 21(2):251--261, 2009.

\bibitem{jutila:divisor}
M.~Jutila.
\newblock On exponential sums involving the divisor function.
\newblock {\em J. Reine Angew. Math.}, 355:173--190, 1985.

\bibitem{jutila:lectures}
M.~Jutila.
\newblock {\em Lectures on a {M}ethod in the {T}heory of {E}xponential {S}ums},
  volume~80 of {\em Tata Institute of Fundamental Research Lectures on
  Mathematics and Physics}.
\newblock Published for the Tata Institute of Fundamental Research, Bombay,
  1987.

\bibitem{jutila:ramanujan}
M.~Jutila.
\newblock On exponential sums involving the {R}amanujan function.
\newblock {\em Proc. Indian Acad. Sci. Math. Sci.}, 97(1-3):157--166 (1988),
  1987.

\bibitem{jutimoto:acta}
M.~Jutila and Y.~Motohashi.
\newblock Uniform bound for {H}ecke {$L$}-functions.
\newblock {\em Acta Math.}, 195:61--115, 2005.

\bibitem{wilton}
J.~R. Wilton.
\newblock A note on {R}amanujan's arithmetical function $\tau(n)$.
\newblock {\em Proc. Cambridge Philos. Soc.}, 25(II):121--129, 1929.

\end{thebibliography}
\end{document}